\documentclass[11pt]{article}
\usepackage{enumerate,latexsym}
\usepackage{latexsym}
\usepackage{amsmath,amssymb}
\usepackage{graphicx}

\newfont{\bb}{msbm10 at 11pt}
\newfont{\bbsmall}{msbm8 at 8pt}
\def\R{\mathbb{R}}

\def\N{\mathbb{N}}
\def\Z{\mathbb{Z}}

\def\D{\mathbb{D}}

\def\a{{\alpha}}
\def\lc{{\cal L}}

\def\g{{\gamma}}

\def\l{{\lambda}}

\def\de{{\delta}}
\def\be{{\beta}}
\def\ve{{\varepsilon}}

\def\centerbmp#1#2#3{\vskip#2\relax\centerline{\hbox to#1{\special
    {bmp:#3 x=#1, y=#2}\hfil}}}

\newtheorem{theorem}{Theorem}

\newtheorem{remark}{Remark}
\newtheorem{corollary}{Corollary}
\newtheorem{definition}{Definition}

\newtheorem{assertion}{Assertion}

\newenvironment{proof}{\smallskip\noindent{\it Proof.}\hskip \labelsep}
                          {\hfill\penalty10000\raisebox{-.09em}{$\Box$}\par\medskip}

\begin{document}

\begin{title}
{Limit leaves of a CMC lamination are stable}
\end{title}
\begin{author}
{William H. Meeks III\thanks{This material is based upon work for the NSF under
Award No. DMS - 0405836. Any opinions, findings, and conclusions or recommendations
expressed in this publication are those of the authors and do not necessarily
reflect the views of the NSF.} \and Joaqu\'\i n P\' erez \and Antonio
Ros\thanks{Research partially supported by a MEC/FEDER grant no. MTM2007-61775. }, }
\end{author}
\maketitle
\begin{abstract}
Suppose ${\cal L}$ is a lamination of a Riemannian manifold by
hypersurfaces with the same constant mean curvature. We prove that
every limit leaf of ${\cal L}$ is stable for the Jacobi operator.
A simple but important consequence of this result is that the set
of stable leaves of ${\cal L}$ has the structure of a lamination.


\end{abstract}

\section{Introduction.}
In this paper we prove that given a codimension one lamination ${\cal L}$ in a
Riemannian manifold~$N$, whose leaves have a fixed constant mean curvature
(minimality is included), then every limit leaf $L$ of ${\cal L}$ is stable with
respect to the Jacobi operator. Our result is motivated by a partial result of Meeks
and Rosenberg in Lemma~A.1 in~\cite{mr13}, where they proved the stability of $L$
under the constraint that the holonomy representation on any compact subdomain
$\Delta \subset L$ has subexponential growth (i.e., the covering space
$\widetilde{\Delta }$ of $\Delta $ corresponding to the kernel of the holonomy
representation has subexponential area growth). In general, if we assume stability
for a covering space $\widetilde{M}$ of a constant mean curvature (CMC) hypersurface
$M$ in~$N$ and for any connected compact domain $\Delta \subset M$ the related
restricted covering $\widetilde{\Delta}\to \Delta$ has subexponential area growth,
then $M$ is also stable, see Lemma~6.2 in~\cite{mpr13} for a proof using cutoff
functions. However, if the area growth of the covering is exponential over some
compact domain in $M$, then the stability of $\widetilde{M}$ does not imply the
stability of $M$, as can be seen in the  example described in the next paragraph.
The existence of this example makes it clear that the application in~\cite{mr13} of
cutoff functions used to prove the stability of a limit leaf $L$   with holonomy of
subexponential growth cannot be applied to case when the holonomy representation of
$L$ has exponential growth.

Consider a compact surface $\Sigma $ of genus at least two endowed with a metric $g$
of constant curvature $-1$, and a smooth function $f\colon \R \to (0,1 ]$ with
$f(0)=1$ and $-\frac{1}{8}<f''(0)<0$. Then in the warped product metric $f^2\, g+
dt^2$ on $\Sigma \times \R $, each slice $M_c=\Sigma \times \{ c\} $ is a CMC
surface of mean curvature $-\frac{f'(c)}{f(c)}$ oriented by the unit vector field
$\frac{\partial }{\partial t}$, and the stability operator on the totally geodesic
(hence minimal) surface $M_0=\Sigma \times \{ 0\} $ is $L=\Delta
+\mbox{Ric}(\frac{\partial }{\partial t})=\Delta -2f''(0)$, where $\Delta $ is the
laplacian on $M_0$ with respect to the induced metric $f(0)^2g=g$ and Ric denotes
the Ricci curvature of $f^2\, g+ dt^2$. The first eigenvalue of $L$ in the (compact)
surface $M_0$ is $2f''(0)$, hence $M_0$ is unstable as a minimal surface. On the
other hand, the universal cover $\widetilde{M}_0$ of $M_0$ is the hyperbolic plane.
Since the first eigenvalue of the Dirichlet problem for the laplacian in
$\widetilde{M}_0$ is $\frac{1}{4}$, we deduce that the first eigenvalue of the
Dirichlet problem for the Jacobi operator on $\widetilde{M}_0$ is
$\frac{1}{4}+2f''(0)>0$. Thus, $\widetilde{M}_0$ is an immersed stable minimal
surface. Similarly, for $c$ sufficiently small, the CMC surface $M_c$ is unstable
but its related universal cover is stable.

\section{The statement and proof of the main theorem.}
In order to help understand the results described in this paper, we
make the following definitions.

\begin{definition} {\rm
Let $M$ be a complete, embedded hypersurface in a manifold $N$. A
point $p\in N$ is a {\it limit point} of $M$ if there exists a
sequence $\{p_n\}_n\subset M$ which diverges to infinity on $M$ with
respect to the intrinsic Riemannian topology on $M$ but converges in
$N$ to $p$ as $n\to \infty$. Let $L(M)$ denote the set of all limit
points of $M$ in $N$. In particular, $L(M)$ is a closed subset of
$N$ and $\overline{M} -M \subset L(M)$, where $\overline{M}$ denotes
the closure of $M$.}
\end{definition}

\begin{definition}
\label{deflamination} {\rm A {\it codimension one lamination} of a Riemannian
manifold $N^{n+1}$ is the union of a collection of pairwise disjoint, connected,
injectively immersed hypersurfaces, with a certain local product structure. More
precisely, it is a pair $({\mathcal L},{\mathcal A})$ satisfying:
\begin{enumerate}
\item ${\mathcal L}$ is a closed subset of $N$;
\item ${\mathcal A}=\{ \varphi _{\be }\colon \D ^n\times (0,1)\to
U_{\be }\} _{\be }$ is a collection of coordinate charts of $N$
(here $\D ^n$ is the open unit ball in $\R^n$, $(0,1)$ the open unit
interval and $U_{\be }$ an open subset of $N$);
\item For each $\be $, there exists a closed subset $C_{\be }$ of
$(0,1)$ such that $\varphi _{\be }^{-1}(U_{\be }\cap {\mathcal
L})=\D ^n\times C_{\be }$.
\end{enumerate}
We will simply denote laminations by ${\mathcal L}$, omitting the
charts $\varphi _{\be }$ in ${\mathcal A}$. A lamination ${\mathcal
L}$ is said to be a {\it foliation of $N$} if ${\mathcal L}=N$.
Every lamination ${\mathcal L}$ naturally decomposes into a
collection of disjoint connected hypersurfaces, called the {\it
leaves} of ${\mathcal L}$. As usual, the regularity of ${\mathcal
L}$ requires the corresponding regularity on the change of
coordinate charts. Note that if $\Delta \subset {\cal L}$ is any
collection of leaves of ${\cal L}$, then the closure of the union of
these leaves has the structure of a lamination within ${\cal L}$,
which we will call a {\it sublamination.} }
\end{definition}

\begin{definition}\label{definition}
{\rm For $H\in \R $, an {\it $H$-hypersurface} $M$ in a Riemannian manifold $N$ is a
codimension one submanifold of constant mean curvature $H$. A codimension one
{$H$-lamination} $\cal{L}$ of $N$ is a collection of immersed (not necessarily
injectively) $H$-hypersurfaces $\{L_\alpha\}_{\alpha\in I}$, called the {\it leaves}
of ${\cal L}$, satisfying the following properties.
\begin{enumerate}
\item ${\cal L}=\bigcup_{\alpha\in I}\{L_\alpha \}$ is a closed subset of
$N$.
\item If $H=0$, then $\lc$ is a lamination of $N$. In this case, we also call ${\cal
L}$ a {\it minimal lamination.}
\item If $H\neq 0$, then given a leaf $L_\alpha$ of $\cal{L}$
 and given a small disk $\Delta \subset L_{\alpha }$, there exists an $\ve
>0$ such that if $(q,t)$ denote the normal coordinates for $\exp _q(t\eta _q)$
(here $\exp $ is the exponential map of $N$ and $\eta $ is the unit normal vector
field to $L_{\a}$ pointing to the mean convex side of $L_{\a }$), then:
\begin{enumerate}
\item The exponential map $\exp \colon U(\Delta ,\ve )=\{ (q,t)\ | \ q\in \mbox{Int}(\Delta ), t\in (-\ve ,\ve )\} $
is a submersion.
\item The inverse image $\exp^{-1}({\cal L})\cap \{ q\in \mbox{Int}(\Delta ), t\in [0,\ve )\} $ is a lamination
of $U(\Delta ,\ve $).
\end{enumerate}
\end{enumerate}}
\end{definition}

The reader not familiar with the subject of minimal or CMC laminations should think
about a geodesic $\g$ on a Riemannian surface. If $\g$ is complete and embedded (a
one-to-one immersion), then its closure is a geodesic lamination $\lc$ of the
surface. When the geodesic $\g$ has no accumulation points, then it is proper.
Otherwise, there pass complete embedded geodesics in $\lc$ through the accumulation
points of $\g$ forming the leaves of $\lc$. A similar result is true for a complete,
embedded $H$-hypersurface of locally bounded second fundamental form (bounded in
compact extrinsic balls) in a Riemannian manifold $N$, i.e., the closure of a
complete, embedded $H$-hypersurface of locally bounded second fundamental form has
the structure of an $H$-lamination of $N$. For the sake of completeness, we now give
the proof of this elementary fact in the case $H\neq 0$ (see the beginning of
Section~1 in~\cite{mr8} for the proof in the minimal case).

Consider a complete, embedded $H$-hypersurface $M$ with locally bounded second
fundamental form in a manifold $N$. Choose a limit point $p$ of $M$ (if there are no
such limit points, then $M$ is proper and it is an $H$-lamination of $N$ by itself),
i.e., $p$ is the limit in $N$ of a sequence of divergent points $p_n$ in $M$. Since
$M$ has bounded second fundamental form near $p$ and $M$ is embedded, then for some
small $\ve >0$, a subsequence of the intrinsic $\ve$-balls $B_M(p_n, \ve)$ converges
to an embedded $H$-ball $B(p,\ve)\subset N$ of intrinsic radius $\ve$ centered at
$p$. Since $M$ is embedded, any two such limit balls, say $B(p,\ve)$, $B'(p,\ve)$,
do not intersect transversally. By the maximum principle for $H$-hypersurfaces, we
conclude that if a second ball $B'(p,\ve)$ exists, then $B(p,\ve)$, $B'(p,\ve)$ are
the only such limit balls and they are oppositely oriented at $p$.

Now consider any sequence of embedded balls $E_n$ of the form $B(q_n,\frac{\ve}{4})$
such that $q_n$ converges to a point in $B(p,\frac{\ve}{2})$ and such that $E_n$
locally lies on the mean convex side of $B(p,\ve)$. For $\ve$ sufficiently small and
for $n$, $m$ large, $E_n$ and $E_m$ must be graphs over domains in $B(p,\ve)$ such
that when oriented as graphs, they have the same mean curvature. By the maximum
principle, the graphs $E_n$ and $E_m$ are disjoint or equal. It follows that near
$p$ and on the mean convex side of $B(p,\ve)$, $\overline{M}$ has the structure of a
lamination with leaves of the same constant mean curvature as $M$. This proves that
$\overline{M}$ has the structure of an $H$-lamination of codimension one.

\begin{definition}
{\rm Let ${\cal L}$ be a codimension one $H$-lamination of a manifold $N$ and $L$ be
a leaf of~${\cal L}$. We say that $L$ is a {\it limit leaf} if $L$ is contained in
the closure of ${\cal L}-L$. }
\end{definition}
We claim that a leaf $L$ of a codimension one $H$-lamination ${\cal L}$ is a limit
leaf if and only if for any point $p\in L$ and any sufficiently small intrinsic ball
$B\subset L$ centered at $p$, there exists a sequence of pairwise disjoint balls
$B_n$ in leaves $L_n$ of ${\cal L}$ which converges to $B$ in $N$ as $n\to \infty $,
such that each $B_n$ is disjoint from $B$. Furthermore, we also claim that the
leaves $L_n$ can be chosen different from $L$ for all~$n$. The implication where one
assumes that $L$ is a limit leaf of ${\cal L}$ is clear. For the converse, it
suffices to pick a point $p\in L$ and prove that $p$ lies in the closure of ${\cal
L}-L$. By hypothesis, there exists a small intrinsic ball $B\subset L$ centered at
$p$ which is the limit in $N$ of pairwise disjoint balls $B_n$ in leaves $L_n$ of
${\cal L}$, as $n\to \infty $. If $L_n\neq L$ for all $n\in \N$, then we have done.
Arguing by contradiction and after extracting a subsequence, assume $L_n=L$ for all
$n\in \N$. Choosing points $p_n\in B_n$ and repeating the argument above with $p_n$
instead of $p$, one finds pairwise disjoint balls $B_{n,m}\subset L$ which converge
in $N$ to $B_n$ as $m \to \infty $. Note that for $(n_1,m_1)\neq (n_2,m_2)$, the
related balls $B_{n_1,m_1}, B_{n_2,m_2}$ are disjoint. Iterating this process, we
find an uncountable number of such disjoint balls on $L$, which contradicts that $L$
admits a countable basis for its intrinsic topology.

\begin{definition}
{\rm A minimal hypersurface $M\subset N$ of dimension $n$ is said to
be {\it stable} if for every compactly supported normal variation of
$M$, the second variation of area is non-negative. If $M$ has
constant mean curvature $H$, then $M$ is said to be {\it stable} if
the same variational property holds for the functional $A-nHV$,
where $A$ denotes area and $V$ stands for oriented volume. A {\it
Jacobi function} $f\colon M \to \R $ is a solution of the equation
$\Delta f+|A|^2f+\mbox{Ric}(\eta )f=0$ on $M$; if $M$ is two-sided,
then the stability of $M$ is equivalent to the existence of a
positive Jacobi function on $M$ (see Fischer-Colbrie~\cite{fi1}). }
\end{definition}

The proof of the next theorem is motivated by a well-known
application of the divergence theorem to prove that every compact
domain in a leaf of an oriented, codimension one minimal foliation
in a Riemannian manifold is area-minimizing in its relative
$\Z$-homology class. For other related applications of the
divergence theorem, see~\cite{ror1}.

\begin{theorem}
\label{stabthm} The limit leaves of a codimension one $H$-lamination of a Riemannian
manifold are stable.
\end{theorem}
\begin{proof}
We will assume that the dimension of the ambient manifold $N$ is
three in this proof; the arguments below can be easily adapted to
the $n$-dimensional setting. The first step in the proof is the
following result.

\begin{assertion}
\label{ass1.5} Suppose $\overline{D}(p,r)$ is a compact, embedded $H$-disk in $N$
with constant mean curvature $H$ (possibly negative), intrinsic diameter $r>0$ and
center $p$, such that there exist global normal coordinates $(q,t)$ based at points
$q\in \overline{D}(p,r)$, with $t\in [0,\ve ]$. Suppose that $T\subset [0,\ve ]$ is
a closed disconnected set with zero as a limit point and for each $t\in T$, there
exists a function $f_t:\overline{D}(p,r)\to [0,\ve ]$ such that the normal graphs
$q\mapsto \exp _q(f_t(q)\eta (q))$ define pairwise disjoint $H$-surfaces with
$f_t(p)=t$, where $\eta $ stands for the oriented unit normal vector field to
$\overline{D}(p,r)$. For each component $(t_{\a }, s_{\a })$ of $[0,\ve ]-T$ with
$s_{\a }<\ve $, consider the interpolating graphs $q\mapsto \exp _q(f_t(q)\eta
(q))$, $t\in [t_{\a },s_{\a }]$, where
\[
f_t= f_{t_{\a}}+(t-t_{\a })\frac{f_{s_{\a }}-f_{t_{\a }}}{s_{\a }-t_{\a }}.
\]
(See Figure~\ref{figure1}).
Then, the mean curvature functions $H_t$ of the graphs of $f_t$ satisfy
\[
\lim _{t\to 0^+}\frac{H_t(q)-H}{t}=0\quad \mbox{for all }q\in D(p,\ve /2).
\]
\end{assertion}
\begin{figure}
\begin{center}
\includegraphics[width=10.8cm]{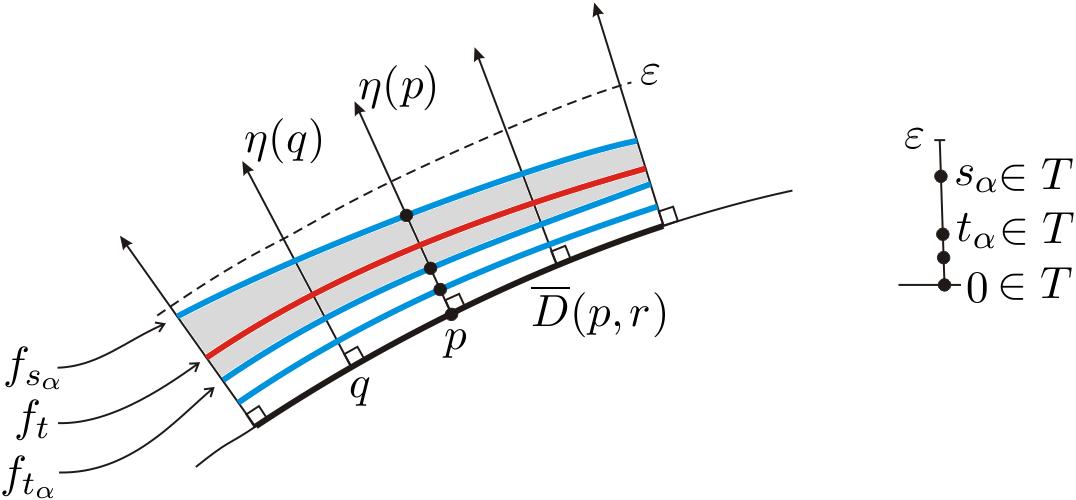}
\caption{The interpolating graph of $f_t$ between the $H$-graphs of $f_{t_{\a
}},f_{s_{\a }}$.}\label{figure1}
\end{center}
\end{figure}
{\it Proof of Assertion~\ref{ass1.5}.} Reasoning by contradiction,
suppose there exists a sequence $t_n\in [0,\ve ]-T$, $t_n\searrow
0$, and points $q_n\in \overline{D}(p,r/2)$, such that
$|H_{t_n}(q_n)-H|>Ct_n$ for some constant $C>0$. Let $(t_{\a
_n},s_{\a _n})$ be the component of $[0,\ve ]-T$ which contains
$t_n$. Then, we can rewrite $f_{t_n}$ as
\[
f_{t_n}=t_n\left[ \frac{t_{\a _n}}{t_n}\frac{f_{t_{\a _n}}}{t_{\a _n}}+
\left( 1-\frac{t_{\a _n}}{t_n}\right)
\frac{f_{s_{\a _n}}-f_{t_{\a _n}}}{s_{\a _n}-t_{\a _n}}\right] .
\]
After extracting a subsequence, we may assume that as $n\to \infty
$, the sequence of numbers $\frac{t_{\a _n}}{t_n}$ converges to some
$A\in [0,1]$, and the sequences of functions $\frac{f_{t_{\a
_n}}}{t_{\a _n}},\frac{f_{s_{\a _n}}-f_{t_{\a _n}}}{s_{\a _n}-t_{\a
_n}}$ converge smoothly to Jacobi functions $F_1,F_2$ on
$\overline{D}(p,r/2)$, respectively. Now consider the normal
variation of $\overline{D}(p,r/2)$ given by
\[
\widetilde{\psi }_t(q)
=\exp _q\left( t[AF_1+(1-A)F_2](q)\eta (q)\right) ,
\]
for $t>0$ small. Since $AF_1+(1-A)F_2$ is a Jacobi function, the
mean curvature $\widetilde{H}_t$ of $\widetilde{\psi _t}$ is
$\widetilde{H}_t=H+{\cal O}(t^2)$, where ${\cal O}(t^2)$ stands for
a function satisfying $t{\cal O}(t^2)\to 0$ as $t\to 0^+$. On the
other hand, the normal graphs of $f_{t_n}$ and of
$t_n(AF_1+(1-A)F_2)$ over $\overline{D}(p,r/2)$ can be taken
arbitrarily close in the $C^4$-norm for $n$ large enough, which
implies that their mean curvatures $H_{t_n},\widetilde{H}_{t_n}$ are
$C^2$-close. This is a contradiction with the assumed decay of
$H_{t_n}$ at $q_n$.
{\hfill\penalty10000\raisebox{-.09em}{$\Box$}\par\medskip}

We now continue the proof of the theorem. Let $L$ be a limit leaf of an
$H$-lamination ${\cal L}$ of a manifold $N$ by hypersurfaces. If $L$ is one-sided,
then we consider the two-sided 2:1 cover $\widetilde{L}\to L$ and pullback  the
$H$-lamination ${\cal L}$ to a small neighborhood of the zero section
$\widetilde{L}_0$ of the normal bundle $\widetilde{L}^{\perp}$ to $\widetilde{L}$
($\widetilde{L}_0$ can be identified with $\widetilde{L}$ itself). In this case, we
will prove that $\widetilde{L}_0$ is stable, which in particular implies stability
for $L$, see Remark~\ref{remark1}. Hence, in the sequel we will assume $L$ is
two-sided.

Arguing by contradiction, suppose there exists an unstable compact
subdomain $\Delta \subset L$ with non-empty smooth boundary
$\partial \Delta $. Given a subset $A\subset \Delta $ and
$\ve >0$ sufficiently small, we define
\[
A^{\perp ,\ve }=\{ \exp _q(t\eta (q))\ | \ q\in A,\ t\in [0,\ve ]\}
\]
to be the one-sided vertical $\ve $-neighborhood of $A$, written in
normal coordinates $(q,t)$ (here we have picked the unit normal
$\eta $ to $L$ such that $L$ is a limit of leaves of ${\cal L}$ at
the side $\eta $ points into). Since ${\cal L}$ is a lamination and
$\Delta $ is compact, there exists $\de  \in (0,\ve )$ such that the
following property holds:
\par
\vspace{.2cm} \noindent {\em $(\star )$ Given an intrinsic disk
$D(p,\de )\subset L$ centered at a point $p\in \Delta $ with radius
$\de $, and given a point $x\in {\cal L}$ which lies in $D(p,\de
)^{\perp ,\ve /2}$, then there passes a disk $D_x\subset {\cal L}$
through $x$, which is entirely contained in $D(p,\de )^{\perp ,\ve
}$, and $D_x$ is a normal graph over $D(p,\de )$. }
\par
\vspace{.2cm} Fix a point $p\in \Delta $ and let $x\in {\cal L}\cap
\{ p\} ^{\perp, \ve /2}$ be the  point above $p$ with greatest
$t$-coordinate. Consider the disk $D_x$ given by property $(\star
)$, which is the normal graph of a function $f_x$ over $D(p,\de )$.
Since $\Delta $ is compact, $\ve $ can be assumed to be small enough
so that the closed region given in normal coordinates by $U(p,\de
)=\{ (q,t)\ | \ q\in D(p,\de ), 0\leq t\leq f_x(q)\} $ intersects
${\cal L}$ in a closed collection of disks $\{ D(t)\ | \ t\in T\} $,
each of which is the normal graph over $D(p,\de )$ of a function
$f_t\colon D(p,\de )\to [0,\ve )$ with $f_t(p)=t$, and $T$ is a
closed subset of $[0,\ve /2]$, see Figure~\ref{figure2}.
\begin{figure}
\begin{center}
\includegraphics[width=7.5cm]{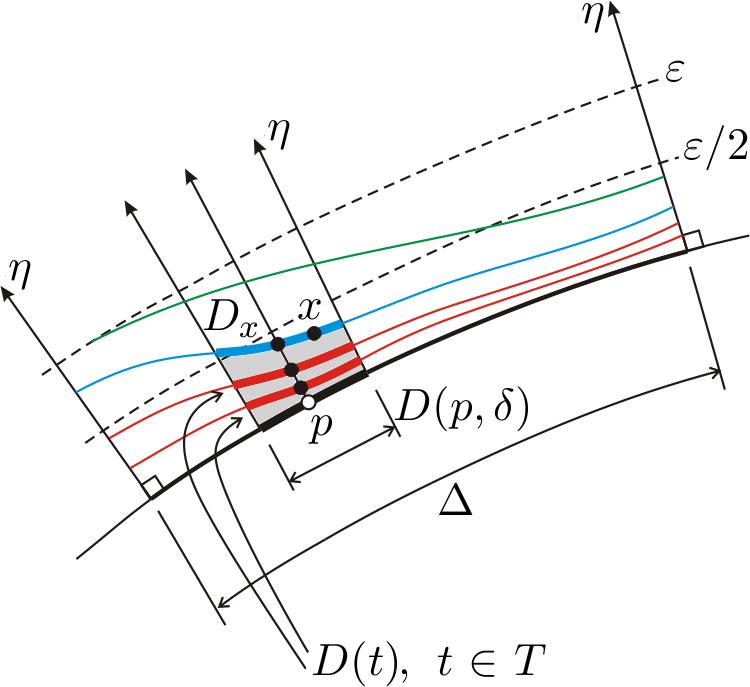}
\caption{The shaded region between $D_x$ and $D(p,\de )$ corresponds to $U(p,\de
)$.}\label{figure2}
\end{center}
\end{figure}
We now foliate the region $U(p,\de )-\bigcup _{t\in T}D(t)$ by
interpolating the graphing functions as we did in
Assertion~\ref{ass1.5}. Consider the union of all these locally
defined foliations ${\cal F}_p$ with $p$ varying in $\Delta $. Since
$\Delta $ is compact, we find $\ve _1\in (0,\ve /2)$ such that
the one-sided normal neighborhood $\Delta ^{\perp ,\ve _1}
\subset \bigcup _{p\in \Delta }{\cal F}_p$
of $\Delta $ is foliated by surfaces which are
portions of disks in the locally defined foliations ${\cal F}_p$. Let
${\cal F}(\ve _1)$ denote this foliation of $\Delta ^{\perp ,\ve _1}$.
By Assertion~\ref{ass1.5}, the mean curvature function of the foliation
${\cal F}(\ve _1)$ viewed locally as a function $H(p,t)$ with $p\in
\Delta $ and $t\in [0,\ve _1]$, satisfies
\begin{equation}
\label{eq:1}
\lim _{t\to 0^+}\frac{H(p,t)-H}{t}=0,\quad \mbox{for all }p\in \Delta .
\end{equation}

On the other hand since $\Delta $ is unstable, the first eigenvalue
$\l _1$ of the Jacobi operator $J$ for the Dirichlet problem on
$\Delta $, is negative. Consider a positive eigenfunction $h$ of $J$
on $\Delta $ (note that $h=0$ on $\partial \Delta $). For $t\geq 0$
small and $q\in \Delta $, $\exp _q(th(q)\eta(q))$ defines a family
of  surfaces $\{ \Delta (t)\} _t$ with $\Delta (t)\subset \Delta
^{\perp ,\ve }$ and the mean curvature $\widehat{H}_t$ of $\Delta
(t)$ satisfies
\begin{equation}
\label{eq:2} \left. \frac{d}{dt}\right|_{t=0}\widehat{H}_t=Jh=-\l
_1h>0\quad \mbox{on the interior of }\Delta .
\end{equation}
Let $\Omega (t)$ be the compact region of $N$ bounded by $\Delta
\cup \Delta (t)$ and foliated away from $\partial \Delta$ by the
surfaces $\Delta (s)$, $0\leq s\leq t$. Consider the smooth unit
vector field $V$ defined at any point $x\in \Omega (t)-\partial
\Delta$ to be the unit normal vector to the unique leaf $\Delta (s)$
which passes through~$x$, see Figure~\ref{figure3}.
\begin{figure}
\begin{center}
\includegraphics[width=9.2cm]{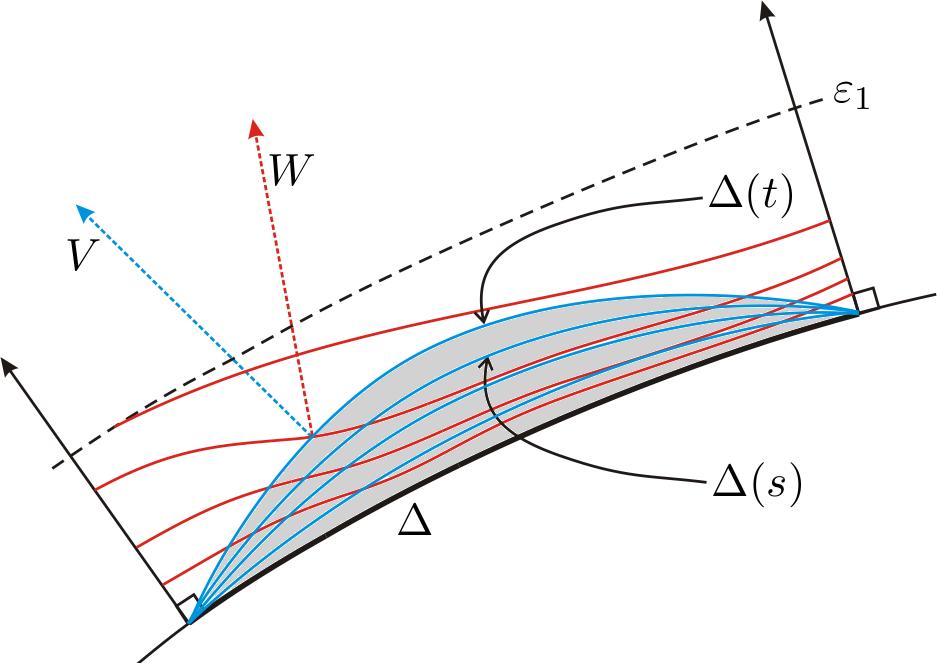}
\caption{The divergence theorem is applied in the shaded region $\Omega (t)$ between
$\Delta $ and $\Delta (t)$.} \label{figure3}
\end{center}
\end{figure}
Since the divergence of $V$ at $x\in
\Delta (s)\subset \Omega (t)$ equals $-2\widehat{H}_s$ where
$\widehat{H}_s$ is the mean curvature of $\Delta (s)$ at $x$, then
(\ref{eq:2}) gives
\[
\mbox{div}(V)=-2\widehat{H}_s=-2H+2\l _1sh+{\cal O}(s^2) \quad
\mbox{on }\Delta (s)
\]
for $s>0$ small. It follows that there exists a positive constant
$C$ such that for $t$ small,
\begin{equation}
\label{eq:3}
\int _{\Omega (t)}\mbox{div}(V)=-2H\mbox{Vol}
(\Omega (t))+2\l _1\int _{\Omega (t)}sh+{\cal O}(t^2)<
-2H\mbox{Vol}(\Omega (t))-Ct.
\end{equation}

Since the foliation ${\cal F}(\ve _1)$ has smooth leaves with
uniformly bounded second fundamental form, then the unit normal
vector field $W$ to the leaves of ${\cal F}(\ve _1)$ is Lipschitz on
$\Delta ^{\perp ,\ve _1}$ and hence, it is Lipschitz on $\Omega
(t)$. Since $W$ is Lipschitz, its divergence is defined almost
everywhere in $\Omega (t)$ and the divergence theorem holds in this
setting. Note that the divergence of $W$ is smooth in the regions of
the form $U(p,\de )-\bigcup _{t\in T}D(t)$ where it is equal to $-2$ times
the mean curvature of the leaves of ${\cal F}_p$. Also, the mean
curvature function of the foliation is continuous on ${\cal F}(\ve _1)$
(see Assertion~\ref{ass1.5}). Hence, the divergence of $W$ can be
seen to be a continuous function on $\Omega (t)$ which equals $-2H$
on the leaves $D(t)$, and by Assertion~\ref{ass1.5}, div$(W)$
converges to the constant $-2H$ as $t\to 0$ to first order. Hence,
\begin{equation}
\label{eq:4}
\int _{\Omega (t)}\mbox{div}(W)>-2H \mbox{Vol}(\Omega (t))-Ct,
\end{equation}
for any $t>0$ sufficiently small.

Applying the divergence theorem to $V$ and $W$ in $\Omega (t)$ (note that
$W=V$ on $\Delta $), we obtain the following two inequalities:
\[
\int _{\Omega (t)}\mbox{div}(V)=\int _{\Delta (t)}\langle V,\eta
(t)\rangle -\int _{\Delta }\langle V,\eta \rangle =
\mbox{Area}(\Delta (t))-\mbox{Area}(\Delta ),
\]
\[
\int _{\Omega (t)}\mbox{div}(W)=\int _{\Delta (t)}\langle W,\eta (t)
\rangle -\int _{\Delta }\langle V,\eta \rangle <
\mbox{Area}(\Delta (t))-\mbox{Area}(\Delta ),
\]
where $\eta (t)$ is the exterior unit vector field to $\Omega (t)$ on
$\Delta (t)$. Hence, $\int _{\Omega (t)}\mbox{div}(W)<\int _{\Omega
(t)}\mbox{div}(V)$. On the other hand, choosing $t$ sufficiently
small such that both inequalities (\ref{eq:3}) and (\ref{eq:4})
hold, we have $\int _{\Omega (t)}\mbox{div}(W)> \int _{\Omega
(t)}\mbox{div}(V)$. This contradiction completes the proof of the
theorem.
\end{proof}

\begin{remark}
\label{remark1} {\rm The proof of the theorem shows that given any two-sided cover
$\widetilde{L}$ of a limit leaf $L$ of ${\cal L}$ as described in the statement of
the theorem, then $\widetilde{L}$ is stable. This follows by lifting ${\cal L}$ to a
neighborhood $U(\widetilde{L})$ of $\widetilde{L}$ in its normal bundle, considered
to be the zero section in $U(\widetilde{L})$. In the case of non-zero constant mean
curvature hypersurfaces, $L$ is already two-sided and then stability is equivalent
to the existence of a positive Jacobi function. However, in the minimal case where a
hypersurface $L$ may be one-sided, this observation concerning stability of
$\widetilde{L}$ is generally a stronger property; for example, the projective plane
contained in projective three-space is a totally geodesic surface which is area
minimizing in its $\Z_2$-homology class but its oriented two-sided cover is
unstable, see Ross~\cite{ro5} and also Ritor\'e and Ros~\cite{rr2}.}
\end{remark}

Next we give a useful and immediate consequence of Theorem~\ref{stabthm}. Let ${\cal
L}$ be a codimension one $H$-lamination of a manifold $N$. We will denote by
$\mbox{\rm Stab}(\lc),\,\mbox{\rm Lim}(\lc)$ the collections of stable leaves and
 limit leaves of $\lc$, respectively. Note that $\mbox{\rm
Lim}(\lc)$ is a closed set of leaves and so, it is a sublamination
of ${\cal L}$.

\begin{corollary}\label{corrs}
Suppose that $N$ is a not necessarily complete Riemannian manifold and ${\cal L}$ is
an $H$-lamination of $N$ with leaves of codimension one. Then, the closure of any
collection of its stable leaves has the structure of a sublamination of ${\cal L}$,
all of whose leaves are stable. Hence,  ${\mbox{\rm Stab}(\lc)}$ has the structure
of a minimal lamination of $N$ and   $\mbox{\rm Lim}(\lc)\subset \mbox{\rm
Stab}(\lc)$ is a sublamination.

\end{corollary}

\begin{remark}
\label{remark2}{\rm  Theorem~\ref{stabthm} and Corollary~\ref{corrs} have many
useful applications to the  geometry of embedded minimal and constant mean curvature
hypersurfaces in Riemannian manifolds. We refer the interested reader to the
survey~\cite{mpe2} by the first two authors and to our joint paper in~\cite{mpr19}
for some of these applications.

}
\end{remark}

\center{William H. Meeks, III at bill@math.umass.edu\\
Mathematics Department, University of Massachusetts, Amherst, MA
01003}
\center{Joaqu\'\i n P\'{e}rez at jperez@ugr.es\\
Department of Geometry and Topology, University of Granada, Granada,
Spain}
\center{Antonio Ros at aros@ugr.es\\
Department of Geometry and Topology, University of Granada, Granada,
Spain}

\bibliographystyle{plain}
\bibliography{bill}

\end{document}